\documentclass[reqno,11pt]{amsart}
\usepackage{amsfonts,amsmath,amsthm, url}
\usepackage{amssymb}  
\usepackage{mathrsfs}
\usepackage{graphics}
\usepackage{graphicx}
\usepackage{color}
\usepackage{stmaryrd} 
   \newtheorem{thm}{Theorem}
   \newtheorem{prop}{Proposition}
   \newtheorem{lem}{Lemma}[section]

   \newtheorem{rem}{Remark}[section]

\newcommand{\N}{\mathbb{N}}

\newcommand{\R}{\mathbb{R}}
\newcommand{\C}{\mathbb{C}}
\newcommand{\E}{\mathbb{E}}

\newcommand{\cH}{\mathcal{H}}
\newcommand{\prob}{\mathbb{P}}

\newcommand{\ds}{\displaystyle}

\newcommand{\vertiii}[1]{{\left\vert\kern-0.25ex\left\vert\kern-0.25ex\left\vert #1 
    \right\vert\kern-0.25ex\right\vert\kern-0.25ex\right\vert}}

\makeatletter
  
  \@addtoreset{equation}{section}
\makeatother
\begin{document}

\title{Stationary martingale solution for the 2D stochastic Gross-Pitaevskii equation}
\author[Anne de Bouard]{Anne de Bouard$^{\scriptsize 1}$}
\author[Arnaud Debussche]{Arnaud Debussche$^{\scriptsize 2,3}$}
\author[Reika Fukuizumi]{Reika FUKUIZUMI$^{\scriptsize 4}$}


\maketitle

\begin{center} \small
$^1$  CMAP, CNRS, Ecole polytechnique, I.P. Paris\\
91128 Palaiseau, France; \\
\email{anne.debouard@polytechnique.edu}
\end{center}

\begin{center} \small
$^2$ Univ Rennes, CNRS, IRMAR - UMR 6625, F-35000 Rennes, France;\\
\end{center}

\begin{center} \small
$^3$ Institut Universitaire de France (IUF);\\
\email{arnaud.debussche@ens-rennes.fr}
\end{center}

\begin{center} \small
$^4$ Research Center for Pure and Applied Mathematics, \\
Graduate School of Information Sciences, Tohoku University,\\
Sendai 980-8579, Japan; \\
\email{fukuizumi@math.is.tohoku.ac.jp}
\end{center}

\section{Introduction} 

In this short report we give a proof of the existence of a stationary solution to 
the following Gross-Pitaevskii equation in $2d$ driven by a space-time white noise:
\begin{equation} \label{eq:SGP}
dX=(\gamma_1 + i\gamma_2)(HX -|X|^{2} X)dt +\sqrt{2\gamma_1}dW, \quad t>0, \quad x \in \R^2,  
\end{equation}
where $H=\Delta-|x|^2$, $\gamma_1 >0$, and $\gamma_2 \in \R$. 
The unknown function $X$ is a complex valued random field on a probability space $(\Omega, \mathcal{F}, \prob)$ endowed 
with a standard filtration $(\mathcal{F}_t)_{t\ge 0}$.
This equation is used as a model for Bose-Einstein condensates in the presence of temperature effects. 
There are some studies in the physics literature using this model \cite{bbdbg,ds,gd,w}. 
We are interested in this equation from a mathematical point of view, and we have studied the $1d$ case in \cite{dbdf0}
 in particular the properties of the statistical equilibrium.  
\vspace{3mm}

As is the case for the stochastic quantization equations \cite{dpd},
the use of renormalization is necessary in order to give a meaning to the solutions of \eqref{eq:SGP}, 
as the Gaussian measure generated by the linear equation is only supported in $\mathcal{W}^{-s,q}$, with $s>0$, $q \ge 2$, and $sq>2$, 
where for $1 \le p \le +\infty$, $\sigma \in \R,$ 
\begin{equation*}
\mathcal{W}^{\sigma,p}(\R^2)=
\{v \in \mathcal{S}'(\R^2), \; |v|_{\mathcal{W}^{\sigma,p}(\R^2)}:=|(-H)^{\sigma/2} v|_{L^p(\R^2)} <+\infty \},
\end{equation*}
denote the Sobolev space associated with the operator $H$.
\vspace{3mm}

Renormalization procedures, using Wick products, have been by now widely used in the context of singular stochastic partial differential
equations, in particular for parabolic equations based on gradient flows (see for example \cite{dpd,tw} for the $2d$ case). 
The complex Ginzburg-Landau equation
driven by space-time white noise, i.e. \eqref{eq:SGP} without the harmonic potential, posed on the three-dimensional torus,
was studied in \cite{h} and for the two-dimensional torus in \cite{m, trenberth}. 
The main difference in our case is the presence of the harmonic potential $|x|^2$. 
\vspace{3mm}

The proof in this report will not be published anywhere, nore in \cite{dbdf1} where the existence of strong global solution for (\ref{eq:SGP}), 
i.e., much stronger result is established. In fact, the proof of this report was our first try for ensuring the existence of a solution to (\ref{eq:SGP}) 
for {\it any} dissipation parameter $\gamma_1$: indeed,  if $\gamma_1$ is sufficiently large we may use simply as in the purely parabolic case \cite{tw} 
an $L^p$ energy estimate to globalize the solution. 
The small $\gamma_1$ case was later solved by using ideas inspired by the bootstrap arguments used
in [10] (although the proof in [10] does not directly applies to the present case), and the proof for any $\gamma_1$ is written in \cite{dbdf1}. 
However, we think it is interesting to present this first proof in this report, with some precisions
 about the dependence of estimates of stationary solutions on the dissipation parameters.

\vspace{3mm}

\section{notation and main result} 

In what follows, we will use the following notation: 
 Let $\{h_k\}_{k \in \N^2}$ be the orthonormal basis of $L^2(\R^2, \R)$, consisting of eigenfunctions of $-H$ 
with corresponding eigenvalues $\{\lambda_k^2\}_{k\in \N^2}$, i.e. $-H h_k=\lambda_k^2 h_k$, $\lambda_k^2=2|k|+2$. 
We take $\{h_k , i h_k \}_{k \in \N^2}$ as a complete orthonormal
system in $L^2(\R^2, \C)$, and we may write the cylindrical Wiener process in \eqref{eq:SGP} as  
\begin{equation} \label{WienerProcess}
W (t, x)= \sum_{k\in \N^2} (\beta_{k,R} (t) + i\beta_{k,I}(t)) h_k (x).
\end{equation}
Here, $(\beta_{k,R} (t))_{t \ge 0}$ and $(\beta_{k,I} (t))_{t \ge 0}$ are sequences of independent real-valued Brownian motions on the stochastic basis $(\Omega, \mathcal{F}, \prob, (\mathcal{F}_t)_{t\ge 0}).$ The notation $\E$ stands for
the expectation with respect to $\prob$.
\vspace{3mm}

We will use an approximation by finite dimensional objects. To verify the convergence properties of a function series of the form 
$u=\sum_{k\in \N^2} c_k h_k,$
we define, for any $N \in \N$ fixed, for any $p \in [1,\infty]$, and $s\in \R$, 
a smooth projection operator $S_N: L^2(\R^2,\C)  \to E_N^{\C} := \mathrm{span}\{h_k\}_{|k|\le N}$ by 
\begin{equation}
\label{def:SN}
S_N \Big[\sum_{k \in \N^2} c_k h_k \Big]:= \sum_{k \in \N^2} \chi \Big[\frac{\lambda_k^2}{\lambda_N^2}\Big] c_k h_k 
=\chi\Big[\frac{-H}{\lambda_N^2}\Big] \Big[\sum_{k \in \N^2} c_k h_k\Big], 
\end{equation}
where $\chi \ge 0$ is a cut-off function such that $\chi \in C_0^{\infty}(-1,1)$, $\chi=1$ on $[-\frac{1}{2}, \frac{1}{2}].$ 
Note that here and in what follows, we denote by $\lambda_N$ the value $\lambda_{(N,0)}$, for simplicity.
The operator $S_N$, which is self-adjoint and commutes with $H$, may be extended by duality to any Sobolev space $\mathcal{W}^{s,2}(\R^2; \C)$,
with $s\in \R$, and thus by Sobolev embeddings, to any space $\mathcal{W}^{s,p}(\R^2; \C)$, with $p\ge 1$.
A simple modification of Theorem 1.1 of \cite{jn} 
implies that $S_N$ is a bounded operator from $L^p$ to $L^p$, uniformly in $N$,  
for any $p \in [1,\infty]$. We denote the usual spectral projector by  
$$\Pi_{N}\Big[\sum_{k \in \N^2} c_k h_k \Big]:= \sum_{k \in \N^2, |k|\le N} c_k h_k,$$ 
which is uniformly bounded only in $L^2$. 
\vspace{3mm} 

Let us recall known facts before mentioning precisely our results. Writing the solution of (\ref{eq:SGP}) as $X=u+Z_{\infty}^{\gamma_1, \gamma_2}$ with  
\begin{equation} \label{eq:Zinfini}
Z_{\infty}^{\gamma_1, \gamma_2}(t)=\sqrt{2\gamma_1}\int_{-\infty}^t e^{(t-\tau) (\gamma_1+i\gamma_2) H} dW(\tau),  
\end{equation}
which is the stationary solution for the linear stochastic equation
\begin{equation} \label{eq:Z}
dZ=(\gamma_1+i\gamma_2)HZdt +\sqrt{2\gamma_1} dW,  
\end{equation}
we find out the following random partial differential equation for $u$:
\begin{equation} \label{eq:u}
\partial_t u = (\gamma_1+i\gamma_2) (Hu-|u+Z_{\infty}^{\gamma_1, \gamma_2}|^2(u+Z_{\infty}^{\gamma_1, \gamma_2})), \quad u(0)=u_0:=X(0)-Z_{\infty}^{\gamma_1, \gamma_2}(0). 
\end{equation}
We are therefore required to solve this random partial differential equation. However, using standard arguments, it is not difficult to see that 
the best regularity we may expect for $Z_\infty^{\gamma_1, \gamma_2}$ is almost surely : 
$Z_{\infty}^{\gamma_1, \gamma_2} \in \mathcal{W}^{-s,q}(\R^2)$ for $s>0$, $q \ge 2,$ $sq>2$ as follows.  
\begin{lem} \label{lem:Kolmogorov} 
Fix any $T>0$. Let $\gamma_1>0, \gamma_2 \in \R, s>0, q \ge 2, sq>2$ and $0<\alpha <\frac 12[(s-\frac2q ) \wedge 1]$. 
The stationary solution $Z_{\infty}^{\gamma_1, \gamma_2}$ of  (\ref{eq:Z}) has a modification in $C^{\alpha}([0,T], \mathcal{W}^{-s,q})$. 
Moreover, there exists a positive constant $C_{T}$ such that 
$$ \E\left[\sup_{t\in [0,T]} |Z^{\gamma_1,\gamma_2}_{\infty}(t) |_{\mathcal{W}^{-s,q}}\right]  \le C_T.$$
\end{lem}
\vspace{3mm}

Thus, to give a sense to the nonlinearity in (\ref{eq:u}), we need a renormalization and we will consider in place 
the renormalized equation of (\ref{eq:u}) :
 \begin{equation} \label{eq:wicku}
\partial_t u = (\gamma_1+i\gamma_2) (Hu-:|u+Z_{\infty}^{\gamma_1, \gamma_2}|^2(u+Z_{\infty}^{\gamma_1, \gamma_2}):), \quad u(0)=u_0:=X(0)-Z_{\infty}^{\gamma_1, \gamma_2}(0), 
\end{equation}
where $:~ :$ in the nonlinear part,  using the notation $Z_{R, \infty}=\mathrm{Re}\,(Z_{\infty}^{\gamma_1, \gamma_2})$, 
$Z_{I, \infty}=\mathrm{Im}\,(Z_{\infty}^{\gamma_1, \gamma_2})$, $u_R=\mathrm{Re} \,u$, and $u_I=\mathrm{Im} \,u$, 
means 
\begin{equation}
\label{def:F}
:|u+Z_{\infty}^{\gamma_1, \gamma_2}|^2 (u+Z_{\infty}^{\gamma_1, \gamma_2}): 
= F\left(u, (:Z_\infty^l:)_{1\le l \le 3}\right) =F_0 +F_1+F_2+F_3
\end{equation}
with $F_0 =|u|^2u$, and
\begin{eqnarray*}
F_1 &= &Z_\infty |u|^2 +2 Z_{R,\infty} u_R u+2 Z_{I,\infty} u_I u,\\
F_2 &= &: Z_{R,\infty}^2: (3u_R+iu_I) \, + :Z_{I,\infty}^2: (u_R + 3iu_I) +2:Z_{R,\infty}Z_{I,\infty}:(u_I+iu_R),\\
F_3& = & :Z_{R,\infty}^3: +\, i:Z_{I,\infty}^3: + :Z_{R,\infty} Z_{I,\infty}^2: + \,i:Z_{R,\infty}^2 Z_{I,\infty}: .
\end{eqnarray*}
Here, for any $k,l \in \N$, the Wick products $:(Z_{R,\infty})^{k} (Z_{I,\infty})^{l}:$ are defined as follows. 
\vspace{3mm}

Recall that the Hermite polynomials $H_n(x)$, $n \in \N$ are defined by 
\begin{equation} \label{eq:HermitePoly}
H_n(x)=\frac{(-1)^n}{\sqrt{n!}} e^{\frac{x^2}{2}} \frac{d^n}{dx^n}(e^{-\frac{x^2}{2}}), \quad n\ge 1
\end{equation}
and $H_0(x)=1$. 
\vspace{3mm}
 
The notation $:(S_N z)^n:(x)$ for $n \in \N$, $N\in \N$, $x\in \R^2$, with a real-valued 
centered Gaussian white noise $z$, means 
$$:(S_N z)^n :(x) = \rho_N(x)^n \sqrt{n !} H_n\left[\frac{1}{\rho_N(x)} S_N z(x)\right], \quad x \in \R^2$$ 
with 
$$\rho_N(x)=\left[\sum_{k \in \N^2} \chi^2\left(\frac{\lambda_k^2}{\lambda_N^2}\right) \frac{1}{\lambda_k^2} (h_k(x))^2\right]^{\frac12}.$$

It is known (see \cite{dbdf0}) that the law $\mathcal{L}(Z_{\infty}^{\gamma_1, \gamma_2})$ 
equals the complex Gaussian measure $\mu=\mathcal{N}_{\C}(0, 2(-H)^{-1}).$
\vspace{3mm}

\begin{prop} \label{prop:ReImCauchy} (\cite{dbdf1})
For any $k,l \in \N$, 
the sequence $\{:(S_N Z_{R,\infty})^{k}: :(S_N Z_{I,\infty})^{l}:\}_{N\in \N}$ is a Cauchy sequence in $L^q(\Omega, \mathcal{W}^{-s,q}(\R^2))$, 
for $q >2$, $s>0$ with $qs >2$.
\vspace{3mm}

Moreover, defining then, for any $k,l\in \N$, for any fixed $t$,
\begin{eqnarray*}
:(Z_{R,\infty})^{k} (Z_{I,\infty})^{l}:~ &:=& \lim_{N\to\infty} :(S_N Z_{R,\infty})^{k}: :(S_N Z_{I,\infty})^{l}:, 
\quad \mbox{in} \quad L^q(\Omega, \mathcal{W}^{-s,q}(\R^2)),
\end{eqnarray*}
where $s>0, q>2$ and $sq>2$, there exists a constant $M_{s,q,k,l}$ such that 
\begin{equation} \label{bound:Z}
\E\left[|:(Z_{R,\infty})^k (Z_{I,\infty})^l:|_{\mathcal{W}^{-s,q}}^q\right] \le M_{s,q,k,l}.
\end{equation}
\end{prop}
\vspace{3mm}

Remark that higher order moments may also be estimated thanks to Nelson formula:
Let $s>0$, $m > q>2$ and $sq>2$.  Then there is a constant $M_{s,q,k,l,m}$ such that
\begin{equation} \label{eq:m-moment} 
\E\left[|:(Z_{R,\infty})^k (Z_{I,\infty})^l:|_{\mathcal{W}^{-s,q}}^m\right] \le M_{s,q,k,l,m}.
\end{equation}
\vspace{3mm} 

Note that if we consider the equation (\ref{eq:wicku}) in terms of $X$, then 
\begin{equation} \label{eq:wickX} 
dX = (\gamma_1 + i\gamma_2)(HX -:|X|^{2} X:)dt +\sqrt{2\gamma_1}dW, \quad t>0, \quad x \in \R^2.
\end{equation}
\vspace{3mm}

In \cite{dbdf1}, we constructed a measure $\rho$ as a weak limit of the family of  finite dimensional Gibbs measure of the form : 
$$d \tilde{\rho}_N (y)= \Gamma_N e^{-\tilde \cH_{N}(S_N y)} dy, \quad y \in E_N^{\C},$$
where $\Gamma_N^{-1}= \int e^{-\tilde \cH_{N}(S_N y)} dy,$ and  
$$\tilde \cH_{N}(y)=\frac{1}{2}|\nabla y|_{L^2}^2 +\frac{1}{2}|x y|_{L^2}^2 + \int_{\R^2}\left[\frac{1}{4}|y(x)|^4 -2\rho_N^2(x)|y(x)|^2+2 \rho_N^4(x)\right]dx.$$
Note that
$$\tilde \cH_{N}(y)=\frac{1}{2}|\nabla y|_{L^2}^2 +\frac{1}{2}|x y|_{L^2}^2 + \frac 14 \int_{\R^2} :|y(x)|^4: dx, $$
and $\nabla_y \tilde \cH_N(y)=-Hy+:|y|^2y:$. 
\vspace{3mm}

\begin{prop} (\cite{dbdf1})
\label{cor:tightness}
The family of finite dimensional Gibbs measures $(\tilde{\rho}_N)_N$ is tight in $\mathcal{W}^{-s,q}$ for any $s>0$, $q \ge 2$ and $sq>2$.  
\end{prop}
\vspace{3mm}
 
 Note that $\tilde{\rho}_N$ does not depend on $\gamma_1$ or $\gamma_2$. It is an invariant measure for  the case of $\gamma_1=0$ and $\gamma_2\ne 0$ or 
for the case of  $\gamma_1>0$ and $\gamma_2 =0$, for the case of $\gamma_1>0$ and $\gamma_2 \ne 0$ to the finite dimensional equation:
 \begin{equation} \label{eq:SGP_finite}
dX=(\gamma_1 + i\gamma_2)(HX -S_N (: |S_N X|^{2} S_N X :)) dt +\sqrt{2\gamma_1} \Pi_N dW, \quad t>0, \quad x \in \R^2, \quad X(0) \in E_N^{\C}.
\end{equation}
The global existence of the solution $X_N$ of this finite dimensional equation is ensured by the standard fixed point methods, and an energy estimate 
as in \cite{dbdf0} (see also (\ref{bound:F}) below). Remark that the invariant measure $\tilde{\rho}_N$ is unique if  $\gamma_1>0$ and $\gamma_2=0$. 
Eq.(\ref{eq:SGP_finite}) gives the Galerkin approximation of $X$ satisfying  (\ref{eq:wickX}). 
\vspace{3mm}  
                                              
Finally we deduce the following main result for the equation \eqref{eq:wickX}:
\begin{thm} \label{thm:main} 
Let $\gamma_1>0$ and $\gamma_2 \in \R$, and 
let $0<s<1$, $q>2$ such that $sq >2$. Then there exists a stationary martingale solution $X$ of  \eqref{eq:wickX}  having trajectories 
in $C(\R_+, \mathcal{W}^{-s,q})$, and $\mathcal{L}(X(t)) = \rho$ for all $t \in \R$.
\end{thm}
The proof of Theorem \ref{thm:main} will be given in Section 3.
\vspace{3mm}

The tightness of the family of measures $(\tilde{\rho}_N)_N$ in Proposition \ref{cor:tightness} was proved in \cite{dbdf1} 
considering the coupled evolution on $E_N^\C$ given by
\begin{equation}
\label{eq:coupled}
\left\{ \begin{array}{rcl} \ds
\frac{du}{dt}& =&(\gamma_1+i\gamma_2)\left[ Hu-S_N\big(:|S_N(u+Z)|^2 S_N (u+Z):\big)\right] \\[0.25cm]
dZ &= & (\gamma_1+i\gamma_2) HZ dt +\sqrt{2\gamma_1}\Pi_N dW,
\end{array} \right.
\end{equation}
one may easily prove, using e.g. similar estimates as in the proof of Proposition \ref{prop:tightness} below, together with the Gaussianity
of $Z$, and a Krylov-Bogolyubov argument, that \eqref{eq:coupled} has an invariant measure $\nu_N$ on $E_N^\C \times \E_N^\C$.
Moreover, by uniqueness of the invariant measure of \eqref{eq:SGP_finite} in case of $\gamma_2=0$, 
we necessarily have for any bounded continuous
function $\varphi$ on $\E_N^\C$ :
$$
\int_{E_N^\C} \varphi(x)\tilde \rho_N(dx)=\int \!\!\int_{E_N^\C\times E_N^\C} \varphi(u+z) \nu_N(du,dz).
$$

\begin{prop}
\label{prop:tightness} Let $\gamma_1>0$ and $\gamma_2 \in \R$. 
Let $(u_N,Z_N) \in C(\R_+;E_N^\C \times E_N^\C)$ be a stationary solution of \eqref{eq:coupled}. Then, for any $m > 0$, there is a constant $C_{m, \gamma_1, \gamma_2}>0$ 
independent of $t$ and $N$, such that
\begin{equation}
\label{eq:h1bound}
\E(|(-H)^{\frac{1}{2m}}u_N|_{L^2}^{2m}) \le C_{m, \gamma_1, \gamma_2}.
\end{equation}
\end{prop}
\vspace{3mm}

\begin{rem}
\begin{itemize}
\item[(1)] The constant $C_{m,\gamma_1, \gamma_2}$ in the RHS does not depend on $\gamma_1$ when $\gamma_2=0$, and if $\gamma_2 \ne 0$, then $C_{m,\gamma_1, \gamma_2}$ depends on the ratio $\frac{|\gamma_2|}{\gamma_1}$. Let $s$ with $0<s<1$ and $q>2$ such that $sq>2$. Using this remark, applying Proposition \ref{prop:tightness} with $m=1$ and $\gamma_2=0$, we deduce that for some positive constant $C$ not depending on $N, \gamma_1$, and for any $t\ge 0$,
$$
\E \big( |u_N(t)|_{\mathcal{W}^{-s,q}}^2 \big) \lesssim \E \big(|u_N(t)|_{L^q}^2\big) \lesssim \E\big(|(-H)^{\frac12}u_N|_{L^2}^2\big) \lesssim C,
$$
where we have used the embedding $\mathcal{W}^{1,2} \subset L^q$, for any $q<+\infty$. Thus,
\begin{eqnarray*}
& & \int_{\mathcal{W}^{-s,q}} |x|^2_{\mathcal{W}^{-s,q}} \tilde \rho_N(dx) =\int\!\! \!\int_{(\mathcal{W}^{-s,q})^2} |u+z|_{\mathcal{W}^{-s,q}}^2 \nu_N(du,dz) \\
& & \le 2\E \big( |u_N(t)|_{\mathcal{W}^{-s,q}}^2 +|Z_N(t)|_{\mathcal{W}^{-s,q}}^2 \big),
\end{eqnarray*}
and the right hand side above is bounded indepently of $N$, $\gamma_1$ and $t$, since the law of $Z_N$ converges 
to a Gaussian measure $\mu$ (which is independent of $\gamma_1$) on ${\mathcal{W}^{-s,q}}$. 
We see thus that the tightness is independent of the parameters $\gamma_1, \gamma_2$, which is consistent with the fact that 
$\tilde{\rho}_N$ is independent of $\gamma_1$ and $\gamma_2.$
\item[(2)] Unfortunately, the bound (\ref{eq:h1bound}) does not provide higher moment bounds on the measures $\tilde \rho_N$, 
 preventing us to obtain  $\rho$-a.s. initial data global existence of solution by the method  in \cite{dpd1}.
\end{itemize}
\end{rem} 
    
\vspace{3mm}

For the proof of Proposition \ref{prop:tightness}, we will use the following interpolation estimates:
Let $\alpha\ge 0$.
\begin{eqnarray} \label{taylor}
|fg|_{\mathcal{W}^{\alpha,q}} \le C(|f|_{L^{q_1}} |g|_{\mathcal{W}^{\alpha, \bar{q_1}}} + |f|_{\mathcal{W}^{\alpha, q_2}} |g|_{L^{\bar{q_2}}}),
\end{eqnarray}
where $1<q<\infty$,  $q_1, q_2 \in (1, \infty]$, $\bar{q_1}, \bar{q_2} \in [1, \infty)$ with 
$\frac{1}{q}=\frac{1}{q_1}+\frac{1}{\bar{q_1}}=\frac{1}{q_2}+\frac{1}{\bar{q_2}}.$
\vspace{3mm}

\noindent
{\em Proof of Proposition \ref{prop:tightness}.}
Taking the $L^2$-inner product of the first equation in \eqref{eq:coupled} with $u_N$ yields
\begin{eqnarray}
\label{bound:F} 
\ds \nonumber& & 
\frac12 \frac{d}{dt} |u_N(t)|_{L^2}^2 +\gamma_1 |(-H)^{\frac12} u_N(t)|_{L^2}^2 +\gamma_1 |S_Nu_N(t)|_{L^4}^4  =  
-\mathrm{Re}(\gamma_1+i\gamma_2) \int_{\R^2} \left[ F_1(S_Nu_N, S_NZ_N)\right.\\
& & \hskip 1 in \left. + \,F_2(S_Nu_N, S_NZ_N)+F_3(S_NZ_N)\right]
\overline{S_Nu_N}(t) dx.
\end{eqnarray}

We first estimate the term containing $F_3$ in the right hand side above. Thanks to Proposition~\ref{prop:ReImCauchy},
taking $0<s<1$ and $q>2$ such that $sq>2$, we may bound
$$
\left| \int_{\R^2} F_3(S_NZ_N)\overline{S_Nu_N} dx \right| \lesssim |F_3(S_NZ_N)|_{\mathcal{W}^{-s,q}} |S_Nu_N|_{\mathcal{W}^{s,q'}}
$$
with $\frac1q +\frac{1}{q'} =1$. Interpolating then $\mathcal{W}^{s,q'}$between $L^r$ and $\mathcal{W}^{1,2}$, with $\frac{1}{q'}=\frac s2 +\frac{1-s}{r}$,
we get
$$
|S_Nu_N|_{\mathcal{W}^{s,q'}} \lesssim |(-H)^{\frac12} S_Nu_N|_{L^2}^s |S_Nu_N|_{L^r}^{1-s}.
$$
On the other hand, noticing that $r\in(1,2)$, we have for any $v\in \mathcal{W}^{1,2}$:
\begin{eqnarray*}
\int_{|x|\ge 1} |v(x)|^r dx & \le & \left[\int_{|x|\ge 1}|x|^2 |v(x)|^2 dx\right]^{\frac r2} \left[ \int_{|x|\ge 1} |x|^{-\frac{2r}{2-r}} dx\right]^{\frac{2-r}{2r}}\\
& \lesssim &|(-H)^{\frac12 } v|_{L^2}^r,
\end{eqnarray*}
so that $|v|_{L^r}\lesssim |(-H)^{\frac12 } v|_{L^2}$. It follows that
\begin{eqnarray}
\label{bound:F3}
\nonumber 
\left| \int_{\R^2} F_3(S_NZ_N) \overline{S_Nu_N}dx\right| 
&  \lesssim &|F_3(S_NZ_N)|_{\mathcal{W}^{-s,q}} |(-H)^{\frac12} S_Nu_N|_{L^2}\\ \nonumber 
&  \le & \frac{4(\gamma_1+|\gamma_2|)}{\gamma_1} |F_3(S_NZ_N)|_{\mathcal{W}^{-s,q}}^2 \\ 
&& \hspace{10mm}                               +\frac{\gamma_1}{4(\gamma_1+|\gamma_2|)} |(-H)^{\frac12} S_Nu_N|_{L^2}^2.
\end{eqnarray}

Next, we consider the term containing $F_2$ in the right hand side of \eqref{bound:F}. First, by \eqref{def:F} and \eqref{taylor},
\begin{eqnarray*}
 \left| \int_{\R^2} F_2(S_Nu_N, S_NZ_N) \overline{S_Nu_N}dx\right|
&  \lesssim &\sum_{l=0}^2 |:S_NZ_{N,R}^l S_NZ_{N,I}^{2-l} :|_{\mathcal{W}^{-s,q}} |(S_N u_N)^2|_{\mathcal{W}^{s,q'}} \\
&  \lesssim &\sum_{l=0}^2 |:S_NZ_{N,R}^l Z_{N,I}^{2-l} :|_{\mathcal{W}^{-s,q}} |S_Nu_N|_{L^4} |S_Nu_N|_{\mathcal{W}^{s,p}}
\end{eqnarray*}
with $\frac{1}{q'}=\frac14 + \frac1p$. Note that $p\in(1,2)$ and we may use the same procedure as before to obtain
$$
|S_Nu_N|_{\mathcal{W}^{s,p}} \lesssim |(-H)^{\frac12} S_Nu_N|_{L^2},
$$
so that 
\begin{eqnarray}
\label{bound:F2}
& & \nonumber \left| \int_{\R^2} F_2(S_Nu_N, S_NZ_N) \overline{S_Nu_N}\right| \\
&  \le  & \frac12 \left(\frac{4(\gamma_1+|\gamma_2|)}{\gamma_1} \right)^3 \sum_{l=0}^2 |:S_NZ_{N,R}^l S_NZ_{N,I}^{2-l}:|_{\mathcal{W}^{-s,q}}^4
 +\, \frac{\gamma_1}{2(\gamma_1+|\gamma_2|)} |S_Nu_N|_{L^4}^4 \\
& & \nonumber + \frac{\gamma_1}{4(\gamma_1+|\gamma_2|)} |(-H)^{\frac12}S_Nu_N|_{L^2}^2.
\end{eqnarray}

We finally turn to the term containing $F_1$ in the right hand side of \eqref{bound:F}. 
We easily get, thanks again to \eqref{taylor},
\begin{eqnarray*}
 \left| \int_{\R^2} F_1(S_Nu_N, S_NZ_N) \overline{S_Nu_N}dx\right|
&  \lesssim &|S_N Z_N|_{\mathcal{W}^{-s,q}} |S_Nu_N|_{L^4}^2 |S_Nu_N|_{\mathcal{W}^{s,r}}
\end{eqnarray*}
where $r>2$ is such that $\frac{1}{q'}=\frac12 + \frac1r$. 
Let $m>2$ with $\frac1r=\frac s2 +\frac{1-s}{m}$, so that
$$
|S_Nu_N|_{\mathcal{W}^{s,r}}\lesssim |(-H)^{\frac12}S_Nu_N|_{L^2}^s |S_Nu_N|_{L^m}^{1-s}.
$$
If $m> 4$, we interpolate $L^m$ between $L^4$ and $L^{2m}$, then use the Sobolev embedding $\mathcal{W}^{1,2} \subset L^{2m}$.
If $2<m\le4$, we interpolate $L^m$ between $L^2$ and $L^4$, then use $\mathcal{W}^{1,2} \subset L^2$.
In both cases we obtain, using in addition the Poincar\'e inequality for $(-H)$:
\begin{eqnarray*}
|S_Nu_N|_{\mathcal{W}^{s,r}} & \lesssim & |(-H)^{\frac12}S_Nu_N|_{L^2}^\alpha |S_Nu_N|_{L^4}^{1-\alpha}
\end{eqnarray*}
for some constant $\alpha\in (0,1)$. We deduce that 
\begin{eqnarray}
& & \nonumber \left| \int_{\R^2} F_1(S_Nu_N, S_NZ_N) \overline{S_Nu_N}\right| \\ \nonumber
&  \lesssim  & |S_NZ_N|_{\mathcal{W}^{-s,q}} |S_Nu_N|_{L^4}^{3-\alpha} |(-H)^{\frac12} S_Nu_N|_{L^2}^\alpha 
\\ \nonumber
& \le &  2^{-\frac{3-\alpha}{1-\alpha}} \left( \frac{4(\gamma_1+|\gamma_2|)}{\gamma_1}\right)^{\frac{1}{2-\alpha} (\alpha+\frac{2(3-\alpha)}{1-\alpha})}  |S_NZ_N|_{\mathcal{W}^{-s,q}}^{\frac{4}{1-\alpha}} \\  \label{bound:F1}
&& \hspace{3mm} + \frac{\gamma_1}{2(\gamma_1+|\gamma_2|)}\left[ |S_Nu_N|_{L^4}^4 
+\frac 12 |(-H)^{\frac12}S_Nu_N|_{L^2}^2\right],
\end{eqnarray}
by Young inequality. 

Gathering \eqref{bound:F}--\eqref{bound:F1}, and noticing that $|(-H)^{\frac12}S_Nu_N|_{L^2}\le |(-H)^{\frac12}u_N|_{L^2}$,
leads to 
\begin{eqnarray*}
 \frac{d}{dt} |u_N(t)|_{L^2}^2 +\frac{\gamma_1}{2} |(-H)^{\frac12} u_N(t)|_{L^2}^2
 &  \lesssim & C_0 \sum_{k+l=1}^3 |:S_NZ_{N,R}^l S_NZ_{N,I}^{k}:|_{\mathcal{W}^{-s,q}}^{m_{k,l}}, 
\end{eqnarray*} 
with 
$$C_0 \le \gamma_1 \left(\frac{\gamma_1}{4(\gamma_1+|\gamma_2|)} \right)^{-\kappa}$$
for some integers $m_{k,l}$ and $\kappa=\kappa_{\alpha}>0$.

Now, let $m>0.$ We multiply by $|u_N(t)|_{L^2}^{2m-2}$ both sides of the above inequality to get 
\begin{eqnarray*}
 \frac{1}{m} \frac{d}{dt} |u_N(t)|_{L^2}^{2m} +\frac{\gamma_1}{2} |(-H)^{\frac12} u_N(t)|_{L^2}^2 |u_N(t)|_{L^2}^{2m-2}
& \lesssim & C_0 |u_N(t)|_{L^2}^{2m-2} \sum_{k+l=1}^3 |:S_NZ_{N,R}^l S_NZ_{N,I}^{k}:|_{\mathcal{W}^{-s,q}}^{m_{k,l}}.
\end{eqnarray*}
Applying the interpolation inequality
\begin{eqnarray*}
|(-H)^{\frac{1}{2m}} u_N|_{L^2} \le C_m |(-H)^{\frac{1}{2}} u_N|_{L^2}^{\frac{1}{m}} |u_N|_{L^2}^{\frac{m-1}{m}}
\end{eqnarray*}
to the second term in the left hand side, and using Young inequality in the right hand side, we obtain
\begin{eqnarray*}
 \frac{1}{m} \frac{d}{dt} |u_N(t)|_{L^2}^{2m} +\frac{\gamma_1}{2 C_m^{2m}} |(-H)^{\frac{1}{2m}} u_N(t)|_{L^2}^{2m}
&  \lesssim & C_0 \varepsilon |u_N(t)|_{L^2}^{2m} + C_0 C_{\varepsilon} \sum_{k+l=1}^3 |:S_NZ_{N,R}^l S_NZ_{N,I}^{k}:|_{\mathcal{W}^{-s,q}}^{m'_{k,l}},
\end{eqnarray*}
for any $\varepsilon>0$ and constants $C_{\varepsilon}>0$ and $m'_{k,l}>0$. 
We choose $C_0 \varepsilon=\frac{\gamma_1 \lambda_1^2}{4 C_m^{2m}}$ after  
using Poincar\'{e} inequality so that the first term of the right hand side is absorbed in the left hand side, 
\begin{eqnarray*}
\frac{1}{m} \frac{d}{dt} |u_N(t)|_{L^2}^{2m} +\frac{\gamma_1}{4 C_m^{2m}} |(-H)^{\frac{1}{2m}} u_N(t)|_{L^2}^{2m}
&  \lesssim & C_0 C_{\varepsilon} \sum_{k+l=1}^3 |:S_NZ_{N,R}^l S_NZ_{N,I}^{k}:|_{\mathcal{W}^{-s,q}}^{m'_{k,l}}.
\end{eqnarray*}
Integrating in time, taking expectations on both sides and using the stationarity of $u_N$ and of the Wick
products, together with \eqref{eq:m-moment}, yields 
$$
\E\big(|(-H)^{\frac{1}{2m}}u_N|_{L^2}^{2m}\big) \lesssim_{m}  \left(\frac{\gamma_1}{4(\gamma_1+|\gamma_2|)} \right)^{-\kappa'}  \sum_{k+l=1}^3 M_{s,q,k,l,m'_{k,l}},
$$ 
for some $\kappa'=\kappa'_{\alpha} >0$ and the conclusion.
\hfill
\qed



\medskip


\section{Proof of Theorem \ref{thm:main}}

We are in position to construct a stationary solution of (\ref{eq:wickX}) for any values of $\gamma_1>0$ and $\gamma_2 \in \R$. 
We have seen in the previous section that the system (\ref{eq:coupled}) has a stationary solution $(u_N, Z_N)$ where 
$Z_N=\Pi_N Z^{\gamma_1, \gamma_2}_{\infty}$. Moreover, it is clear that $(u_N, Z^{\gamma_1, \gamma_2}_{\infty})$ is then a stationary solution of   
\begin{equation}
\label{eq:stat}
\left\{ \begin{array}{rcl} \ds
\frac{du}{dt}& =&(\gamma_1+i\gamma_2)\left[ Hu-S_N\big(:|S_N(u+Z)|^2 S_N (u+Z):\big)\right] \\[0.25cm]
dZ &= & (\gamma_1+i\gamma_2) HZ dt +\sqrt{2\gamma_1} dW.
\end{array} \right.
\end{equation}
In this section we will denote $Z^{\gamma_1, \gamma_2}_{\infty}$ by $Z$ for the sake of simplicity. 
Using Proposition \ref{prop:tightness}, we first prove that the the law of this sequence $\{(u_N, Z)\}_{N \in \N}$ is tight 
in an appropriate space to construct a martingale solution.  
\vspace{3mm}

We will use the following lemma (see \cite{dbdf1} for the proof). 
\begin{lem} \label{lem:bilinear} Let $1<p<q<+\infty$, $0<s<\beta<2/p$, and $m \in \N \setminus \{0\}$. 
Suppose 
$\beta-s-(m-1)(\frac{2}{p}-\beta) > 0$, and  $s+m(\frac{2}{p}-\beta)<2(1-\frac{1}{q})$.
Then, there is a constant $C>0$ such that 
\begin{equation*}
|h f^{m}|_{\mathcal{W}^{-(s+m(\frac{2}{p}-\beta)), q}} \le C|h|_{\mathcal{W}^{-s,q}} |f|_{\mathcal{W}^{\beta,p}}^{m}. 
\end{equation*}
\end{lem}
\vspace{3mm}

\begin{lem} \label{lem:uz_tendu}Let $\gamma_1>0$, $\gamma_2 \in \R$, 
$0<s<1$, and $q>2$ such that $qs>2$. Let also $0<\delta<\frac16$, and let $p> \max \{q, 24, \frac{2}{\delta} \}$. 
The sequence $(u_N, Z)_{N\in \N}$ is bounded in 
$$ L^{2m}(\Omega, L^{2m}(0,T, H^{\frac{1}{m}})) \cap L^{\frac{4}{3}}(\Omega, \mathcal{W}^{1,\frac{4}{3}}(0,T, \mathcal{W}^{-2,p})) \times C^{\alpha}([0,T], \mathcal{W}^{-s,q}\cap \mathcal{W}^{-\delta,p})$$
for any $m > 0$, and $\alpha>0$ satisfying $\alpha<\min(\frac{s}{2}-\frac{1}{q}, \frac{\delta}{2}-\frac1p, 1)$. 
\end{lem}

\begin{proof} 
It suffices to check the bound in $L^{\frac{4}{3}}(\Omega, \mathcal{W}^{1,\frac{4}{3}}(0,T, \mathcal{W}^{-2,p}))$, since the
other bounds follow from Proposition \ref{prop:tightness}, the stationarity of $u_N$, and Lemma \ref{lem:Kolmogorov} for $Z$. 
We write the equation for $u_N$; 
\begin{eqnarray*}
{u}_N(t)  
&=& {u}_N(0) + (\gamma_1+i\gamma_2) \int_0^t H {u}_N (\sigma) d\sigma \\ 
&& -(\gamma_1+i\gamma_2) 
\int_0^t S_N(:|S_N({u}_N+ {Z}_N)|^2 S_N ({u}_N+{Z}_N):) (\sigma) d\sigma.   
\end{eqnarray*}
In the right hand side, the first term is constant and clearly bounded 
in $L^{\frac{4}{3}}(\Omega, \mathcal{W}^{1,\frac{4}{3}}(0,T, \mathcal{W}^{-2,p}))$ 
by Proposition \ref{prop:tightness} and H\"{o}lder inequality. For the second term, we have 
\begin{eqnarray*}
\E \left(\left|\int_0^{\cdot} H u_N(\sigma) d\sigma \right|^{\frac{4}{3}}_{\mathcal{W}^{1,\frac{4}{3}}(0,T, \mathcal{W}^{-2,p})}\right) 
&\le& C_T \E\left(|H u_N|^{\frac{4}{3}}_{L^{\frac{4}{3}}(0,T, \mathcal{W}^{-2,p})} \right) \\
&\le & C_T \E \left(|u_N|^2_{L^2(0,T, \mathcal{W}^{1,2})} \right),
\end{eqnarray*}
where we have used Sobolev embedding and H\"{o}lder inequality in the last inequality. 
To estimate the nonlinear terms, we decompose as in (\ref{def:F}), 
\begin{eqnarray*}
S_N(:|S_N (u_N+Z)|^2 S_N (u_N+Z):)&=&S_N (F_0(S_N u_N)+ F_1(S_N u_N, S_N Z) \\
&& \hspace{3mm} +F_2(S_N u_N, S_N Z)+F_3(S_N Z)).
\end{eqnarray*}
The terms in $F_3$ are simply estimated thanks to Proposition \ref{prop:ReImCauchy} as follows.
$$
\E \left( \left|\int_0^{\cdot} S_N F_3 (s) ds \right|^{\frac{4}{3}}_{\mathcal{W}^{1,\frac{4}{3}}(0,T, \mathcal{W}^{-2,p})}\right) 
\le C_T \E (|S_N F_3|_{L^{\frac{4}{3}}(0,T, \mathcal{W}^{-s,p})}^{\frac{4}{3}}) 
\le C_T M_{s,p}^{\frac{4}{3}}. 
$$
For the term $F_0$, using Sobolev embeddings $L^{\frac43} \subset \mathcal{W}^{-2,p}$ 
and $\mathcal{W}^{\frac12, 2} \subset L^4$, we obtain  
\begin{eqnarray*}
&& \E (|S_N F_0|_{L^{\frac{4}{3}}(0,T, \mathcal{W}^{-2,p})}^{\frac{4}{3}}) 
= \int_0^T \E(||S_N u_N|^2 S_N u_N|^{\frac{4}{3}}_{\mathcal{W}^{-2,p}}) ds \\
& & \le  C_T \E(||S_N u_N|^2 S_N u_N |^{\frac{4}{3}}_{L^{\frac{4}{3}}}) = C_T \E(| S_N u_N|_{L^4}^4) \le C_T \E(|u_N|^4_{\mathcal{W}^{\frac12, 2}}),
\end{eqnarray*}
which is bounded independently of $N$ by Proposition \ref{prop:tightness}.
To estimate the $F_1$-terms, we fix $s'>0$ such that $s'<\frac{1}{12}$ and $s'p>2$, and apply Lemma \ref{lem:bilinear} to get
\begin{eqnarray*}
\E\left(\left|S_N F_1\right|_{L^{\frac{4}{3}}(0,T, \mathcal{W}^{-2,p})}^{\frac{4}{3}}\right) 
\le C_T \E \left(|F_1|_{\mathcal{W}^{-(s'+\frac{14}{24}), p}}^{\frac{4}{3}} \right)
&\le& C_T \E \left(|S_N Z|_{\mathcal{W}^{-s',p}}^{\frac{4}{3}} |S_N u_N|_{\mathcal{W}^{\frac{3}{8},3}}^{\frac{8}{3}} \right).
\end{eqnarray*}
Using then the Sobolev embedding $\mathcal{W}^{\frac{17}{24},2} \subset \mathcal{W}^{\frac{3}{8},3}$ and H\"older inequality, 
the right hand side is majorized by 
\begin{eqnarray*}
C_T \E(|S_N Z|_{\mathcal{W}^{-s',p}}^{24})^{\frac{1}{18}}~ \E(|S_N u_N|^{\frac{48}{17}}_{\mathcal{W}^{\frac{17}{24},2}})^{\frac{17}{18}},  
\end{eqnarray*}
and is thus bounded independently of $N$ thanks to Propositions \ref{prop:ReImCauchy} and \ref{prop:tightness}. 
Finally, for the terms in $F_2$, we apply again Lemma \ref{lem:bilinear}: 
\begin{eqnarray*}
\E\left(\left|S_N F_2\right|_{L^{\frac{4}{3}}(0,T, \mathcal{W}^{-2,p})}^{\frac{4}{3}}\right) 
&\le& T\E \left(|F_2|_{\mathcal{W}^{-(s'+\frac{1}{6}), p}}^{\frac{4}{3}} \right)\\
&\le& C_T \E(\sum_{k+l=2}|:(S_N Z_R)^k (S_N Z_I)^l :|_{\mathcal{W}^{-s',p}}^{\frac{4}{3}} |S_N u_N|_{\mathcal{W}^{\frac13,4}}^{\frac43})\\
&\le & C_T \sum_{k+l=2} \E(|:(S_N Z_R)^k (S_N Z_I)^l :|_{\mathcal{W}^{-s',p}}^{3})^{\frac49}~ \E(|u_N|_{\mathcal{W}^{\frac56,2}}^{\frac{12}{5}})^{\frac59}, 
\end{eqnarray*}
which is bounded again by Propositions \ref{prop:ReImCauchy} and \ref{prop:tightness}. Note that  
in the last inequality we have used the Sobolev embedding $\mathcal{W}^{\frac56, 2} \subset \mathcal{W}^{\frac13,4}$ and H\"older inequality. 
\end{proof}

\begin{rem} \label{rem:compactness}
We note that by Lemma \ref{lem:uz_tendu}, $(u_N)_{N\in \N}$ is bounded in 
\begin{eqnarray*}
&& L^3(\Omega, L^3(0,T, \mathcal{W}^{\frac23,2})) \cap L^{\frac43}(\Omega, \mathcal{W}^{1,\frac43}(0,T, \mathcal{W}^{-2,p})) \\
&& \hspace{3cm} \subset 
L^{\frac43}(\Omega, L^3(0,T, \mathcal{W}^{\frac23,2})) \cap L^{\frac43}(\Omega, \mathcal{W}^{\frac{1}{12},3}(0,T, \mathcal{W}^{-2,p})), 
\end{eqnarray*}
and $L^3(0,T, \mathcal{W}^{\frac23,2}) \cap \mathcal{W}^{\frac{1}{12},3}(0,T, \mathcal{W}^{-2,p})$ is compactly embedded in $L^3(0,T, L^4_x)$.
On the other hand, $(u_N)_{N\in \N}$ is also bounded in 
\begin{eqnarray*}
&& L^2(\Omega, L^2(0,T, \mathcal{W}^{1,2})) \cap L^{\frac43}(\Omega, \mathcal{W}^{1,\frac43}(0,T, \mathcal{W}^{-2,p})) \\ 
&& \hspace{3cm} \subset L^{\frac43}(\Omega, L^2(0,T, \mathcal{W}^{1,2}) \cap \mathcal{W}^{\frac{1}{12},2}(0,T,\mathcal{W}^{-2,p})), 
\end{eqnarray*}
and $L^2(0,T, \mathcal{W}^{1,2}) \cap \mathcal{W}^{\frac{1}{12},2}(0,T,\mathcal{W}^{-2,p})$ is compactly embedded
in $L^2(0,T,\mathcal{W}^{s,2})$ for any $s$ with $0 \le s <1$. In particular, since  
$\mathcal{W}^{\frac56, 2} \subset \mathcal{W}^{\frac13,4}$, the embedding is compact in  
$L^2(0,T, \mathcal{W}^{\frac13,4}).$ 
Finally, we note that $\mathcal{W}^{1,\frac43}(0,T,\mathcal{W}^{-2,p})$ is compactly embedded in 
$C([0,T], \mathcal{W}^{-3,p})$. 
\end{rem}

\begin{proof}[Proof of Theorem \ref{thm:main}] Let $\alpha>0$ satisfy the condition in Lemma \ref{lem:uz_tendu}.  
We deduce from Lemma \ref{lem:uz_tendu}, Remark \ref{rem:compactness} and Markov inequality that the sequence $\{(u_N, Z)_{N\in \N}\}$ is tight in 
\begin{equation} \label{space:compactness}
L^3(0,T, L^4) \cap L^2(0,T, \mathcal{W}^{\frac56,2}) \cap C([0,T], \mathcal{W}^{-3,p}) \times C^{\beta}([0,T], \mathcal{W}^{-s', q}\cap \mathcal{W}^{-\delta,p})
\end{equation}
for any $\beta <\alpha$ and $s'> s$. Fix such $\beta$ and $s'$. By Prokhorov Theorem, there exists a subsequence, still denoted $\{(u_N, Z)_{N\in \N}\}$ which converges 
in law to a measure $\nu$ on the space \eqref{space:compactness}.
By Skorokhod Theorem, there exist $(\tilde{\Omega}, \tilde{\mathcal{F}}, \tilde{\prob})$, $(\tilde{u}_N, \tilde{Z}_N)_{N \in \N}$ and $(\tilde{u}, \tilde{Z})$ 
taking values in the same space \eqref{space:compactness}, satisfying $\mathcal{L}((u_N, Z))=\mathcal{L}((\tilde{u}_N, \tilde{Z}_N))$ for any $N \in \N$, $\mathcal{L}((\tilde{u}, \tilde{Z}))=\nu$, 
and $\tilde{u}_N$ converges to  $\tilde{u}$, $\tilde{\prob}$-a.s. in $L^3(0,T, L^4) \cap L^2(0,T, \mathcal{W}^{\frac56,2}) \cap C([0,T], \mathcal{W}^{-3,p})$, 
$\tilde{Z}_N$  converges to $\tilde{Z}$, $\tilde{\prob}$-a.s. in $C^{\beta}([0,T], \mathcal{W}^{-s', q}\cap \mathcal{W}^{-\delta,p})$. 
Moreover, by diagonal extraction, it can be assumed that this holds for any $T>0$. It is easily seen that $(\tilde{u}, \tilde{Z})$ is a stationary process thanks to 
the convergence of $(\tilde{u}_N, \tilde{Z}_N)$ to $(\tilde{u}, \tilde{Z})$ in $C([0,T], \mathcal{W}^{-3,p}) \times C^{\beta}([0,T], \mathcal{W}^{-s', q})$. This convergence also implies
$\mathcal{L}(Z)=\mathcal{L}(\tilde{Z})$. Note also that if we extract from the subsequence we took for the weak convergence $\tilde{\rho}_N$ to $\rho$, 
we have that 
for each $t\in \R$, $\mathcal{L}(\tilde X(t) )=\mathcal{L}(\tilde u (t)+\tilde Z(t))=\rho$.

Write then, 
\begin{eqnarray} \nonumber
\tilde{u}_N(t)- \tilde{u}_N(0) 
&=& (\gamma_1+i\gamma_2) \int_0^t H \tilde{u}_N (\sigma) d\sigma \\ \label{eq:uz_N}
&& -(\gamma_1+i\gamma_2) 
\int_0^t S_N(:|S_N(\tilde{u}_N+ \tilde{Z}_N)|^2 S_N (\tilde{u}_N+\tilde{Z}_N):) (\sigma) d\sigma.   
\end{eqnarray}
It remains us to show that the right hand side of \eqref{eq:uz_N} converges, up to a subsequence, to 
\begin{eqnarray*} 
(\gamma_1+i\gamma_2) \int_0^t H \tilde{u} (\sigma) d\sigma  
-(\gamma_1+i\gamma_2) 
\int_0^t :|(\tilde{u}+ \tilde{Z})|^2 (\tilde{u}+\tilde{Z}): (\sigma) d\sigma,  
\end{eqnarray*}
$\tilde{\prob}$-a.s. in $C([0,T], \mathcal{W}^{-2,p})$. This can be checked as follows. First, the convergence of 
the linear term follows from the convergence of $\tilde{u}_N$ to $\tilde{u}$ 
in $L^2(0,T, \mathcal{W}^{\frac56,2}) \subset L^1(0,T, \mathcal{W}^{-2,p})$. In order to prove the convergence 
of nonlinear terms, we again decompose the nonlinear terms into the four terms $F_0, \dots,F_3$ as in (\ref{def:F}) and estimate them separately, i.e.,
\begin{eqnarray*}
&& \sup_{t\in [0,T]} \left|\int_0^t :|\tilde{u}+\tilde{Z}|^2 (\tilde{u}+\tilde{Z}): (\sigma)d\sigma 
-\int_0^t S_N(:|S_N(\tilde{u}_N+ \tilde{Z}_N)|^2 S_N (\tilde{u}_N+\tilde{Z}_N):) (\sigma) d\sigma  \right|_{\mathcal{W}^{-2,p}}\\ 
& \le & \sum_{k=0}^3 \int_0^T |F_k(\tilde{u}, \tilde{Z})-S_N F_k(S_N \tilde{u}_N, S_N \tilde{Z}_N)|_{\mathcal{W}^{-2,p}} d\sigma \\
&=& I_0^N+I_1^N+I_2^N+I_3^N.
\end{eqnarray*}
We begin with the convergence of $I_0^N$. 
\begin{eqnarray*}
I_0^N &=& \int_0^T |S_N(|S_N \tilde{u}_N|^2 S_N \tilde{u}_N) -|\tilde{u}|^2 \tilde{u}|_{\mathcal{W}^{-2,p}} d\sigma\\
&\le& \int_0^T ||S_N \tilde{u}_N|^2 S_N \tilde{u}_N - |\tilde{u}|^2 \tilde{u}|_{L^{\frac43}} d\sigma 
+ \int_0^T |(S_N -I) |\tilde{u}|^2 \tilde{u}|_{L^{\frac43}} d\sigma. 
\end{eqnarray*}
Since $\tilde{u}_N $ converges to $\tilde{u}$ a.s. in $L^3(0,T, L^4_x)$, by dominated convergence, the same holds for $S_N \tilde{u}_N$,
thus $|S_N \tilde{u}_N|^2 S_N \tilde{u}_N$ converges to $|\tilde{u}|^2 \tilde{u}$ a.s. in $L^{\frac{4}{3}}(0,T, L^{\frac43})$. 
Similarly, the second term converges to zero by dominated convergence, 
therefore, $I_0^N$ converges to $0$. Next, we use Lemma \ref{lem:bilinear}
to obtain
\begin{eqnarray*}
I_1^N &\le& \int_0^T |F_1(\tilde{u}, \tilde{Z})-S_N F_1(S_N \tilde{u}_N, S_N \tilde{Z}_N)|_{\mathcal{W}^{-(\delta+\frac13),p}} d\sigma \\
&\lesssim& |\tilde{Z} - S_N \tilde{Z}_N|_{C([0,T],\mathcal{W}^{-\delta,p})} \left(|S_N \tilde{u}_N|^2_{L^2(0,T, \mathcal{W}^{\frac13,4})} + 
|\tilde{u}|^2_{L^2(0,T, \mathcal{W}^{\frac13,4})}\right) \\
&& \hspace{3mm}+  \left(|S_N \tilde{Z}_N|_{C([0,T], \mathcal{W}^{-\delta,p})} + 
|\tilde{Z}|_{C([0,T], \mathcal{W}^{-\delta,p})}\right)
\left(|S_N \tilde{u}_N|_{L^2(0,T, \mathcal{W}^{\frac13,4})} + 
|\tilde{u}|_{L^2(0,T, \mathcal{W}^{\frac13,4})}\right) \\
&& \hspace{5cm} \times |S_N \tilde{u}_N - \tilde{u}|_{L^2(0,T, \mathcal{W}^{\frac13,4})}.
\end{eqnarray*}
Hence, $I_1^N$ converges to $0$  since $S_N \tilde{Z}_N$ converges to $\tilde{Z}$ in $C([0,T], \mathcal{W}^{-\delta,p})$ and, by Remark \ref{rem:compactness},
$S_N \tilde{u}_N$ converges to $\tilde{u}$ in $L^2(0,T, \mathcal{W}^{\frac13,4})$. 
Concerning $F_2$, we proceed as for $F_1$, and we use Lemma \ref{lem:bilinear} in the same way;
\begin{eqnarray*}
I_2^N &\lesssim& \sum_{k+l=2} (|:\tilde{Z}_R^k \tilde{Z}_I^l :|_{L^2(0,T, \mathcal{W}^{-\delta,p})} +|S_N(:(S_N \tilde{Z}_{N,R})^k (S_N \tilde{Z}_{N,I})^l :)|_{L^2(0,T, \mathcal{W}^{-\delta,p})} ) \\
&& \times |S_N \tilde{u}_N -\tilde{u}|_{L^2(0,T, \mathcal{W}^{\frac13,4})} \\
&& + \sum_{k+l=2}  |:\tilde{Z}_R^k \tilde{Z}_I^l :- S_N(: (S_N \tilde{Z}_{N,R})^k (S_N \tilde{Z}_{N,I})^l :) |_{L^2(0,T, \mathcal{W}^{-\delta,p})} \\
&& \times \left(|S_N \tilde{u}_N|_{L^2(0,T, \mathcal{W}^{\frac13,4})} + |\tilde{u}|_{L^2(0,T, \mathcal{W}^{\frac13,4})} \right).
\end{eqnarray*}
Note that Proposition \ref{prop:ReImCauchy} implies the convergence to zero of the second term in $L^p(\Omega)$, thus extracting a subsequence, 
$\tilde{\prob}$-a.s.  The first term goes to $0$, too since, again, $S_N \tilde{u}_N$ converges to $\tilde{u}$ in $L^2(0,T, \mathcal{W}^{\frac13,4})$. 
The term $I_3^N$ can be treated similarly. 
We deduce from this convergence result that $\tilde u$ satisfies
$$
\tilde{u}(t)- \tilde{u}(0) 
= (\gamma_1+i\gamma_2) \int_0^t H \tilde{u} (\sigma) d\sigma  -(\gamma_1+i\gamma_2) 
\int_0^t :|(\tilde{u}+ \tilde{Z})|^2 (\tilde{u}+\tilde{Z}):) (\sigma) d\sigma.   
$$
Since it is clear that $\tilde Z$ satisfies the second equation in \eqref{eq:stat}, we easily deduce that $\tilde X=\tilde u +\tilde Z$ is a stationary solution of
\eqref{eq:wickX}  on $(\tilde \Omega, \tilde F, \tilde \prob)$. Moreover, it is not difficult to prove that $\tilde u$ is continuous with values in 
$\mathcal{W}^{-s,q}$, which ends the proof of Theorem \ref{thm:main}.
\end{proof}
\vspace{3mm}

{\bf Acknowledgements.} 
This work was supported by JSPS KAKENHI Grant Numbers JP19KK0066, JP20K03669. 
A. Debussche is partially supported by the French government thanks to the "Investissements d'Avenir" program ANR-11-LABX-0020-0, 
Labex Centre Henri Lebesgue.


\begin{thebibliography}{99}

\bibitem{bbdbg} P.B. Blakie, A.S. Bradley, M.J. Davis,
R.J. Ballagh and C.W. Gardiner, 
\newblock 
{``Dynamics and statistical mechanics of ultra-cold Bose gases using c-field techniques,''}
\newblock
{Advances in Physics.}
{\bf 57} (2008) no.5, 363-455.

\bibitem{dbdf0}A. de Bouard, A. Debussche and R. Fukuizumi, 
\newblock
{``Long time behavior of Gross Pitaevskii equation at positive temperature,''}
\newblock
{SIAM. J. Math. Anal.} {\bf 50} (2018) 5887-5920.

\bibitem{dbdf1} A. de Bouard, A. Debussche and R. Fukuizumi, 
\newblock
{``Two dimensional Gross-Pitaevskii equation with space-time white noise,''} 
\newblock
{submitted.}

  
\bibitem{dpd} G. Da Prato and A. Debussche,
\newblock
{``Strong solutions to the stochastic quantization equations,''}
\newblock
{The Annals of Probability.}
{\bf 31} (2003) no.4, 1900-1916.

\bibitem{dpd1} G. Da Prato and A. Debussche, 
\newblock
{``Two-dimensional Navier-Stokes equations driven by a space-time white noise,''}
\newblock 
{J. Funct. Anal.} 
{\bf 196} (2002) 180-210..


\bibitem{ds} R. Duine and H. Stoof, 
\newblock
{``Stochastic dynamics of a trapped Bose-Einstein condensate,''}
\newblock
{Phys. Rev. A} {\bf 65} (2001) p. 013603.


\bibitem{gd} C.W. Gardiner and M.J. Davis,
\newblock 
{``The stochastic Gross-Pitaevskii equation: II,''}
\newblock
{J. Phys. B} {\bf 36} (2003) 4731-4753.

\bibitem{h} M. Hoshino, 
\newblock
{``Global well-posedness of complex Ginzburg-Landau equation with a space-time white noise,''}
\newblock
{Ann. l'institut Henri Poincar\'{e} (B) Probability and Statistics.} {\bf  54} (2018) 1969-2001.

\bibitem{jn} A. Jensen and S. Nakamura, 
\newblock
{``$L^p$-mapping properties of functions of Schr\"odinger 
operators and their applications to scattering theory,''}
\newblock
{J. Math. Soc.Japan.} 
{\bf 47} (1995) 253-273.


\bibitem{m} T. Matsuda, 
\newblock
{``Global well-posedness of the two-dimensional stochastic complex Ginzburg-Landau equation with cubic nonlinearity,''}
\newblock 
arXiv: 2003.01569.



\bibitem{taylor} M. E. Taylor, 
\newblock
{``Tools for PDEs,''}
\newblock 
{Pseudodifferential Operators, Paradifferential Operators, and Layer Potentials, Math. Surveys Monogr., vol. 81,
American Mathematical Society, Providence, RI} (2000)

\bibitem{trenberth} W. J. Trenberth
\newblock {``Global well-posedness for the two-dimensional stochastic complex Ginzburg-Landau equation"}
\newblock arXiv:1911.09246v1.

\bibitem{tw} P. Tsatsoulis and H. Weber, 
\newblock
{``Spectral gap for the stochastic quantization equation on the
2-dimensional torus,''} 
\newblock
{Ann. l'IHP. Probabilit\'{e}s et statistiques} {\bf 54} (2018) 1204-1249.

\bibitem{w} C.N. Weiler et al. 
\newblock 
{``Spontaneous vortices in the formation of Bose-Einstein condensates,''}
\newblock
{Nature} 
{\bf 455} (2008) nature 07334. 


\end{thebibliography}
\end{document}